\documentclass{amsart}

\usepackage[latin1]{inputenc}
\usepackage{amsfonts}
\usepackage{amsmath}
\usepackage{amsthm}
\usepackage{amssymb}
\usepackage{latexsym}
\usepackage{enumerate}

\newtheorem{theorem}{Theorem}
\newtheorem{lemma}[theorem]{Lemma}
\newtheorem{proposition}[theorem]{Proposition}

\newtheorem{question}[theorem]{Question}

\newtheorem{definition}[theorem]{Definition}

\newtheoremstyle{note}
{3pt}
{3pt}
{}
{}
{\bf}
{:}
{.5em}
{}

\theoremstyle{note}
\newtheorem{remark}[theorem]{Remark}

\numberwithin{equation}{section}

\begin{document}

\newcommand{\cc}{\mathfrak{c}}
\newcommand{\N}{\mathbb{N}}
\newcommand{\BB}{\mathbb{B}}
\newcommand{\C}{\mathbb{C}}
\newcommand{\Q}{\mathbb{Q}}
\newcommand{\R}{\mathbb{R}}
\newcommand{\T}{\mathbb{T}}
\newcommand{\st}{*}
\newcommand{\PP}{\mathbb{P}}
\newcommand{\rin}{\right\rangle}
\newcommand{\SSS}{\mathbb{S}}
\newcommand{\forces}{\Vdash}
\newcommand{\dom}{\text{dom}}
\newcommand{\osc}{\text{osc}}
\newcommand{\F}{\mathcal{F}}
\newcommand{\A}{\mathcal{A}}
\newcommand{\B}{\mathcal{B}}
\newcommand{\I}{\mathcal{I}}
\newcommand{\X}{\mathcal{X}}
\newcommand{\Y}{\mathcal{Y}}
\newcommand{\Z}{\mathcal{Z}}
\newcommand{\CC}{\mathcal{C}}

\thanks{The author was partially supported by the NCN (National Science
Centre, Poland) research grant no.\ 2020/37/B/ST1/02613.}
\thanks{The author would like to thank H.M. Wark for comments on a preliminary version
of this paper  which allowed the subsequent versions to be improved.}

\subjclass[2010]{46B20, 46E15, 03E75, 46B26, 54D80}
\title{Banach spaces in which large subsets of spheres concentrate}
\author{Piotr Koszmider}
\address{Institute of Mathematics of the Polish Academy of Sciences,
ul. \'Sniadeckich 8,  00-656 Warszawa, Poland}
\email{\texttt{piotr.koszmider@impan.pl}}

\begin{abstract} 
We construct a nonseparable Banach space $\X$ (actually, of density continuum) such that 
any uncountable subset $\Y$ of the unit sphere of $\X$ contains uncountably many
points distant by  less than $1$ (in fact, by less then $1-\varepsilon$ for some $\varepsilon>0$).
This solves in the negative the central problem of the search for a nonseparable version 
of Kottman's theorem which so far has produced many 
deep positive results for special classes of Banach spaces and has related 
the global properties of the spaces to the distances between points  of uncountable subsets of 
the unit sphere. The property
of our space is strong enough to imply that it contains neither an uncountable Auerbach system nor
an uncountable equilateral set. The space is a strictly convex renorming of the Johnson-Lindenstrauss
space induced by an $\R$-embeddable almost disjoint family of subsets of $\N$. We also show that
 this special feature of the almost disjoint family is essential to obtain the above properties.
\end{abstract}

\maketitle

\section{Introduction}

All Banach spaces considered in this paper are infinite dimensional and over the reals.
For unexplained terminology see Section \ref{preliminaries}. 
 If $\X$ is a Banach space,  $\Y\subseteq \X$ and $r>0$, we say that $\Y$ is 
\begin{itemize}
\item $(r+)$-separated if $\|y-y'\|>r$, for any two distinct
$y, y'\in \Y$,
\item $r$-separated if $\|y-y'\|\geq r$, for any two distinct
$y, y'\in \Y$,
\item $r$-concentrated if $\|y-y'\|\leq r$, for any two distinct
$y, y'\in \Y$,
\item $r$-equilateral if $\|y-y'\|=r$ for any two distinct
$y, y'\in \Y$,
\item equilateral if it is $r$-equilateral for some $r>0$.
\end{itemize}
The classical Riesz Lemma of 1916 (If $\Y$ is a closed proper subspace of $\X$ and $\varepsilon>0$,
 then there is $x$ in the unit sphere of $\X$ such that the distance of $x$ from $\Y$ does
  exceed $1-\varepsilon$,  \cite{riesz})
allows one to construct $(1-\varepsilon)$-separated sets  of the cardinality equal to the density of a Banach space  and in its unit sphere.
By the compactness of balls in finite dimensional spaces it also yields infinite $1$-separated sets
is any infinite dimensional Banach space.
Kottman proved in 1975 (\cite{kottman}) that the unit sphere of 
every infinite dimensional Banach space admits an infinite
 $(1+)$-separated subset, which was improved in 1981  by Elton and Odell to $(1+\varepsilon)$-separated for
 some $\varepsilon>0$ (\cite{elton-odell}) who also noted
 that $c_0(\omega_1)$ does not admit an uncountable
 $(1+\varepsilon)$-separated set. The Kottman constant of a Banach space (the supremum over
 $\delta>0$ such that there is an infinite $\delta$-separated subset of the unit sphere)
 turned out to be an important tool used for investigating the geometry of the space
 (e.g., \cite{castillo, kryczka, unifconvex, papini1, papini2}).
 In fact, it is related to many aspects of Banach spaces such as, e.g., 
 packing balls, measures of incompactness, fixed points, average
distances  and infinite dimensional convexity (see e.g., the papers citing  \cite{kottman2} or
\cite{elton-odell}).
 
 It has been clear, at least since the paper \cite{mer-ck}  of Mercourakis and Vassiliadis,  that  
 the nature of separation  in uncountable subsets of the unit sphere of a nonseparable
 Banach space could be equally indicative of the global and diverse properties
  of the space as in the separable case. 
  The question of whether the unit sphere of a nonseparable Banach space must contain an uncountable
  $(1+)$-separated set, and if so, of what cardinality compared to the density of the space, has been studied
 for various classes of Banach spaces 
 e.g., in \cite{cuth, tt, mer-pams, mer-ck} and recently culminated in the
  paper \cite{hajek-tams}, where it is highlighted as a central question.  Notably,
 the existence of $(1+\varepsilon)$-separated sets of the size equal to the density of the
 space was proved for super-reflexive Banach spaces by
 Kania and Kochanek in \cite{tt} and for $C(K)$ spaces,
  where $K$ is compact Hausdorff and totally disconnected  by Mercourakis and Vassiliadis. 
  Moreover  H\'ajek, Kania and Russo proved  in \cite{hajek-tams}  that 
 uncountable $(1+)$-separated sets exist in any Banach space of sufficiently large density.
However, we provide quite a strong negative answer to the general question:
 
 \begin{theorem}\label{main} There is a (strictly convex) Banach space of density continuum where every uncountable subset
 of the unit sphere of regular cardinality $\kappa$ includes a  subset of the same cardinality  which is $(1-\varepsilon)$-concentrated for
 some $\varepsilon>0$.
 \end{theorem}
 To prove it apply  Propositions \ref{Rembed}
 and \ref{reduction} obtaining the validity
 of their hypotheses by   Lemmas \ref{exists-ad}  and \ref{operator} respectively.
  The strength of the above property\footnote{Another striking property of such spaces
  is that one cannot pack uncountably many pairwise disjoint 
  open balls of radius ${1\over3}$ into the unit ball, as such a packing would yield an
  uncountable $1$-separated subset of the unit sphere
  (cf. \cite{kottman2}). Packing of uncountably many
  balls of radius ${1\over3}-\varepsilon$ into the unit ball is possible in any inseparable Banach space for
  any $\varepsilon>0$ by 
  the Riesz lemma.} may be appreciated by seeing  how easily it yields
  the next two theorems which provide other natural properties of Banach spaces unknown  to occur until now without
  making some additional consistent but unprovable set-theoretic assumptions.

  Recall an observation of Terenzi from \cite{terenzi} that if $\Y$ is an equilateral set
  in a Banach space $\X$, by scaling we may assume that it is a 
  $1$-equilateral set, and then by considering $\{y_0-y: y\in \Y\setminus\{y_0\}\}$ for
  any $y_0\in \Y$ we obtain a 
  $1$-equilateral set  consisting of elements of the unit sphere of $\X$. So Theorem \ref{main}
  yields:
  
  \begin{theorem}\label{main-equi} There is a (strictly convex) Banach space of density 
  continuum which does not admit an uncountable equilateral set.
\end{theorem}

Banach spaces with this property  were previously consistently constructed by the author
in \cite{equi}. These spaces were of the form $C(K)$ for $K$ compact Hausdorff, while
it was also proved in \cite{equi} that consistently every Banach space of the form $C(K)$ admits
an uncountable equilateral set. 

Different absolute examples of Banach spaces satisfying Theorem \ref{main-equi} are being presented at the same time in a joint paper of the author with H.M. Wark \cite{equi2}. They are renormings of 
$\ell_1([0,1])$ and moreover they do not admit even any infinite equilateral set
(However all equivalent renormings of $\ell_1(\Gamma)$ for $\Gamma$ uncountable admit
 $(1+\varepsilon)$-separated set of size $\Gamma$ by Remark 3.16 of \cite{hajek-tams}).
The above theorem solves Problem 293 of  \cite{guirao}.

As in the case of $(1+)$-separated or $(1+\varepsilon)$-separated sets
the existence of infinite or uncountable equilateral sets (by the above
argument of Terenzi one may assume that such sets are subsets of the unit sphere) was proved by various authors
for many particular classes of Banach spaces. For example, infinite
equilateral sets exist in any Banach space which contains an isomorphic copy of
$c_0$ (\cite{mer-pams}) or any uniformly smooth space (\cite{unif-convex}).
However, in contrast to the case  of $(1+)$-separated sets there are infinite
dimensional Banach spaces with no infinite equilateral subsets (\cite{terenzi, terenzi2, mer-serdica}).
Uncountable equilateral sets in nonseparable Banach spaces have been investigated as well, for
example, in \cite{cuth, equi, mer-pams, mer-ck}. Also the strict convexity
of the norm of our example should be compared with the results saying
that   the unit sphere of  an infinite dimensional uniformly convex
space admits a
 $(1+\varepsilon)$-separated set of the cardinality equal to the density of the space
(\cite{unifconvex}, Proposition 4.16 of \cite{hajek-tams}).

As stressed in \cite{hajek-tams} $(1+)$-separated subsets of the sphere are related to Auerbach bases. 
  Recall
 that for a Banach space $\X$ the system $(x_i, x^*_i)_{i\in \I}\subseteq \X\times \X^*$ is called a biorthogonal system when $x^*_j(x_i)=0$ if $i\not=j$ and $x^*_i(x_i)=1$ for each $i, j\in I$. 
 It is called an Auerbach system if it is biorthogonal and $\|x_i\|=1=\|x^*_i\|$ for every $i\in I$.
 It is clear that the elements of $\X$ in an Auerbach system $(x_i, x^*_i)_{i\in \I}$ form a $1$-separated subset of the unit sphere as
 $\|x_i-x_j\|\geq |x_i(x_i-x_j)|=|x_i^*(x_i)|=1$ for any distinct $i, j\in I$. Thus 
  Theorem \ref{main} yields the following:
  
\begin{theorem} There is a (strictly convex) Banach space of density continuum with no uncountable
Auerbach system.
\end{theorem}
 
 A Banach space with this property was previously only constructed under the assumption
 of the Continuum Hypothesis {\sf CH} in \cite{hajek-tams} (in fact, this is a renorming of
 $c_0(\omega_1)$ so WLD, the property not shared by our space).
 By a result of Day every separable Banach space admits an infinite  Auerbach system (\cite{day}).
 Constructions of Banach spaces of  density continuum with no fundamental
 Auerbach systems were presented in \cite{godun2, godun, plicko, troyanski}.

 Our Banach space is an equivalent renorming $(\X_\A, \|\ \|_T)$ of the subspace 
 $(\X_\A, \|\ \|_\infty)$ of $\ell_\infty$ which is spanned by
 $c_0$ and characteristic functions of elements of an uncountable  almost disjoint family $\A$
 of infinite subsets of $\N$. Recall that $\A\subseteq \wp(\N)$ is called almost disjoint
 if $A\cap A'$ is finite for any distinct $A, A'\in\A$.
 Such spaces $\X_\A$
 were first considered by Johnson and Lindenstrauss in \cite{jl, jl-cor}.
  The renorming $\|\ \|_T$ is obtained in a standard
 way using a bounded injective operator $T:\ell_\infty\rightarrow \ell_2$ (see Section \ref{renormings}). 
 As is well known such a space $(\X_\A, \|\ \|_\infty)$ is isometric to
 the space of the form $C_0(K_\A)$, 
 where $K_\A$ is locally compact, totally disconnected, scattered, separable
 Hausdorff space known as $\Psi$-space or Mr\'owka-Isbell  space or Alexandroff-Urysohn space
 (see e.g., \cite{hrusak-ad, supermrowka}). In particular, it is $c_0$-saturated (by \cite{pel}),
 so admits infinite equilateral sets by \cite{mer-pams}.
The main point, which actually allows one to
 conclude directly the previous theorems, is to obtain the behaviour
 of separated sets of the unit sphere  as in $c_0(\Gamma)$ and at the same time
 to have the underlying locally compact space separable (which allows the construction an injective
 operator from $C(K_\A)$ into the separable $\ell_2$). This is 
 summarized in the following theorem which is proved by applying
 Lemma \ref{exists-ad} and Proposition \ref{Rembed}:
 
 \begin{theorem} There is a separable locally compact Hausdorff space $K$
 of weight continuum such that for every $\varepsilon>0$ and every   subset $\Y$  of
 the unit sphere of $C_0(K)$ of a regular uncountable cardinality there is a subset
 $\Z\subseteq \Y$  of the same cardinality which is
 $(1+\varepsilon)$-concentrated.
 \end{theorem}
 
 To obtain the above  property of the space $(\X_\A, \|\ \|_\infty)$ and
 consequently the properties of its renorming $(\X_\A, \|\ \|_T)$ we need 
 some special property of the almost disjoint family $\A$. A known property
 that is sufficient for us is the $\R$-embeddability of $\A$ (Definition \ref{def-Rembed}).
 On the other hand we show that certain almost disjoint families which are not $\R$-embeddable,  known as
 Luzin families (Definition \ref{def-luzin}), induce the
  Banach space $\X_\A$  such that the sphere of $(\X_\A, \|\ \|_\infty)$ admits 
 an uncountable $2$-equilateral subset (Proposition \ref{luzin-2}) 
 and the sphere of $(\X_\A, \|\ \|_T)$ admits 
 an uncountable $(1+ {1\over 2})$-separated subset (Proposition \ref{luzin-renorm}). For more
 on $\R$-embeddability see \cite{akemann}. Note that A. Dow showed in \cite{dow} that assuming the proper forcing axiom {\sf PFA} 
 every maximal almost disjoint family contains a Luzin subfamily. However
 $\R$-embeddable families  
 of cardinality continuum are abundant and elementary
   to construct with no additional set-theoretic assumptions (like Luzin families of cardinality $\omega_1$).

 Finally let us comment on the isomorphic theory structure of our spaces. 
 $\X_\A$ in the norm $\|\  \|_\infty$ admits a  $1$-equilateral set of cardinality continuum 
 which is an Auerbach system (just characteristic functions
 of the elements of the almost disjoint family $\A$). Although there are many
 nonisomorphic Banach spaces of the form $\X_\A$ for an almost disjoint families 
 $\A$ (\cite{sailing, mar-pol})
 under Martin's Axiom  {\sf MA} and the negation of {\sf CH} all spaces $\X_\A$ are
 pairwise isomorphic for $\A$s of cardinality $\omega_1$, in particular
 $\X_\A$ can be  isomorphic to $\X_{\A'}$ with $\A$ being $\R$-embeddable and
 $\A'$ being Luzin (\cite{sailing}). This, for example,  provides equivalent renormings 
 of $(\X_A, \|\ \|_\infty)$ for any Luzin family $\A$  which satisfy Theorem \ref{main}.
 In fact, at least consistently,
 every nonseparable Banach space can be renormed so that the new unit sphere
 admits uncountable $2$-equilateral sets
  (Theorem 3 \cite{mer-pams}).
 It is unknown at the present moment if this can be proved in ZFC alone.

 \section{Preliminaries and terminology}\label{preliminaries}
 
 \subsection{Notation}\label{notation} 
 
 The notation attempts to be standard. In particular  $C(K)$ denotes the set
 of all continuous functions on a compact Hausdorff space $K$ and
 $C_0(K)$ denotes all continuous functions $f$ on a locally compact 
 Hausdorff $K$ such that for each $\varepsilon>0$
 there is a compact $L\subset K$ such that $|f(x)|<\varepsilon$
 for all $x\in K\setminus L$. Both of these types of linear spaces are considered as Banach
 spaces with the supremum norm which is denoted by $\|\ \|_\infty$. The notation $1_A$ denotes the characteristic function 
 of a set $A$.  When $\Y$ is a subset of a Banach space,
 $\overline{span}(\Y)$ denotes the norm closure of the linear span of $\Y$. The density
 of an infinite dimensional Banach space is the minimal cardinality of a norm dense subset
 of the space. By $\omega_1$ we mean
 the first uncountable cardinal. All almost disjoint families 
 of subsets of $\N$ considered in this paper are uncountable and consist of infinite sets.
 A cardinal $\kappa$ is said to be of uncountable cofinality if the union of countably many sets
 of cardinalities smaller than $\kappa$ has cardinality smaller than $\kappa$. If $A$
 is a set, then $[A]^2$ denotes the family of all two-element subsets of $A$.
The cardinals like $\omega_1$ or the continuum are of uncountable cofinality (\cite{jech}).
 If $(x_n)_{n\in \N}$ is  a sequence of reals and $A\subseteq \N$ is infinite by  $\lim_{n\in A}x_n$ we mean 
 $\lim_{k\rightarrow\infty}x_{n_k}$, where $({n_k})_{k\in \N}$ is any (bijective) renumeration of
 $A$.
 
 \subsection{Banach spaces $\X_\A$}

Given an infinite almost disjoint family $\A$ of infinite subsets of $\N$ we consider the subspace
of $(\ell_\infty, \|\ \|_\infty)$ defined by
$$\X_\A=\overline{span}(\{1_A: A\in \A\}\cup  c_0).$$
Note that  for a fixed $A_0\in\A$ the set
$$\ell_\infty^{A_0}=\{f\in \ell_\infty: \lim_{n\in A_0}f(n) \ \hbox{exists}\}$$
is a closed (in the $\|\ \|_\infty$-norm) linear subspace of $\ell_\infty$. As for
every $A_0\in\A$ all the generators $\{1_A: A\in \A\}\cup  c_0$ of $\X_\A$ are
in this space, it follows that the entire $\X_\A$ is contained in
every $\ell_\infty^{A}$ for $A\in \A$.

\begin{lemma}\label{functions-ad} Suppose that $\A$ is an almost disjoint family of
infinite subsets of $\N$. Let  $k\in \N$, $g\in c_0$,  $A_1, \dots, A_k\in \A$ be distinct,
$q_1, \dots q_k\in \R$ and
$$f=g+\sum_{1\leq i\leq k} q_i 1_{A_{i}}.$$
Then for every $1\leq j \leq k$ we have $\lim_{n\in A_j}f(n)=q_j$.
\end{lemma}
\begin{proof} We have $\lim_{n\in A_j}g(n)=0$ as $g\in c_0$ and
$\lim_{n\in A_j}1_{A_i}(n)=0$ for $i\not=j$ since $A_i\cap A_j$ is finite.
\end{proof}

 \subsection{Equivalent renormings induced by bounded operators}\label{renormings}
 
 Recall that two norms $\|\ \|$ and $\|\ \|'$ on a Banach space $\X$ are equivalent if
 the identity operator between $(\X, \|\ \|)$ and $(\X, \|\  \|')$ is an isomorphism, that is, when
 there are positive constants $c, C$ such that
 $c\| x\|\leq \|x \|'\leq C\| x\|$ for every $x\in \X$. For more on renormings
 see \cite{dgz}.
 It is clear that if  $\X$ and $\Y$ are Banach spaces with
 norms $\|\ \|_\X$ and $\|\ \|_\Y$ respectively and  $T:\X\rightarrow \Y$ is
 a bounded linear operator. Then 
 $$\|x\|_T=\|x \|_\X+\|T(x)\|_\Y$$
 is a norm on $\X$ which is equivalent to the norm $\|\ \|_\X$. In this paper
 besides the supremum norm  $\|\ \|_\infty$ we will consider norms of the   form $\|\ \|_T$.
 Recall that a norm $\| \ \|$ on a Banach space $\X$ is called strictly convex if
 $\|x+y\|=\|x\|+\|y\|$ for $x, y\in \X\setminus\{0\}$
 implies that $x=\lambda y$ for some $\lambda> 0$.
 An example of a strictly convex norm is the standard $\| \ \|_2$ norm
 on $\ell_2$.
 
 \begin{lemma}\label{sconvex} Suppose that $\X$ and $\Y$ are Banach spaces.
 If $T: \X\rightarrow \Y$ is an injective  bounded linear operator, 
 $(\Y, \|\ \|_\Y)$ is strictly
 convex, then $(\X, \|\ \|_T)$ is strictly convex.
 \end{lemma}
 \begin{proof} For non-zero $x, y\in \X$, by the triangle inequality, the condition
 $$\|x+y\|_\X+\|T(x)+T(y)\|_\Y=\|x\|_\X+\|T(x)\|_\Y+\|y\|_\X+\|T(y)\|_\Y$$
  implies that
 $\|x+y\|_\X=\|x\|_\X+\|y\|_\X$ and
 $\|T(x)+T(y)\|_\Y=\|T(x)\|_\Y+\|T(y)\|_\Y$. The latter
 implies  $T(x)=\lambda T(y)$ for some $\lambda> 0$, which gives 
 $x=\lambda y$ since $T$ is injective.
 \end{proof}

\begin{lemma}\label{operator}
There is an injective bounded linear operator $T:\ell_\infty\rightarrow\ell_2$.
In particular, the equivalent renorming $(\ell_\infty, \|\ \|_T)$ is strictly convex.
\end{lemma}
\begin{proof} Consider $T:\ell_\infty \rightarrow\ell_2$ given by
$$T(f)=\Big({f(n)\over 2^n}\Big)_{n\in \N}.$$
$\|f\|_\infty\leq \|f\|_\infty+\|T(f)\|_2=
 \|f\|_\infty+\sqrt{\sum_{n\in\N}{f(n)^2\over 2^n}}\leq(1+\sqrt2)\| f\|_\infty$.
\end{proof}

 \begin{lemma}\label{separable} 
 Suppose that $\kappa$ is a cardinal of uncountable cofinality and $\Y$  
is a subset of cardinality $\kappa$
 of a separable Banach space $\X$. Then for every $\varepsilon>0$ there is $\Y'\subseteq \Y$
 of cardinality $\kappa$ such that $\|y-y'\|\leq\varepsilon$ for every $y, y'\in \Y'$. 
 \end{lemma}
 \begin{proof} Let $\{x_n: n\in \N\}$ be norm dense countable subset of $\X$. Balls of diameter
 $\varepsilon$ with the centers  $x_n$s cover $\X$ and so cover $\Y$. By the uncountable
 cofinality of $\kappa$ one of the balls must contain $\kappa$ elements of $\Y$.
 \end{proof}

 \begin{lemma}\label{spheres} Suppose that $0<a\leq\|x\|\leq \|x'\|\leq b<1$ for some $a, b\in \R$ and $x, x'$
 in a Banach space $\X$. Then
 $$\|x-x'\|\leq b\Bigg\|{x\over{\|x\|}}-{x'\over{\|x'\|}}\Bigg\|+(b-a)$$
 \end{lemma}
 \begin{proof}
 $$\|x-x'\|\leq \|x-(\|x\|/\|x'\|)x'\|+\|(\|x\|/\|x'\|)x'-x'\|\leq$$
 $$\leq \|x\|\Bigg\|{x\over{\|x\|}}-{x'\over{\|x'\|}}\Bigg\| 
 +\Bigg|{\|x\|\over{\|x'\|}} -1 \Bigg|\|x'\|\leq$$
  $$\leq b\Bigg\|{x\over{\|x\|}}-{x'\over{\|x'\|}}\Bigg\|+(1-a/b)b$$
 \end{proof}
 
 The following reduction result  allows us to infer the main properties 
 of the final space $(\X_\A, \|\ \|_T)$ as stated in Theorem \ref{main} from
 the properties of the space $(\X_\A, \|\ \|_\infty)$ which are proved in Proposition \ref{Rembed}.
 
 \begin{proposition}\label{reduction}
 Suppose that $\mathcal X$ is a nonseparable Banach space and $\kappa$ is a
 cardinal  of uncountable cofinality
 such that for every subset $\Y$ of the unit sphere of 
 $(\X, \|\ \|_\X)$ of cardinality $\kappa$  and every $\varepsilon>0$
 there is $\Y'\subseteq \Y$ of cardinality $\kappa$ 
  which is $(1+\varepsilon)$-concentrated. If $T: \X\rightarrow \Z$ is a bounded linear injective operator
  into a Banach space $\Z$ with separable range, then $(\X, \|\ \|_T)$ has the following property:
  For every subset $\Y$ of  the unit sphere of $(\X, \|\ \|_T)$ of cardinality $\kappa$  
 there is $\delta>0$  and $\Y'\subseteq \Y$ of cardinality $\kappa$  
  which is $(1-\delta)$-concentrated.
 \end{proposition}
 \begin{proof} Let $\Y$  be a subset of the unit sphere of $(\X, \|\ \|_T)$  of cardinality
 $\kappa$. As $T$ is injective we have $0<\|y\|_\X<\|y\|_T=1$ for every $y\in \Y$.
 The interval $(0,1)$ can be covered by intervals of the form $(q-(1-q)/4, q)$ for
 rationals $q$ satisfying $0<q<1$. As $\kappa$ has uncountable cofinality by passing to a subset of cardinality $\kappa$ we may assume that there is a rational $0<q<1$ such that
 $$q-(1-q)/4<\|y\|_\X< q$$
 for every $y\in \Y$. Now apply the property of $\X$ with the norm $\|\ \|_\X$
 to $\{y/\|y\|_\X: y\in \Y\}$ and $\varepsilon=(1-q)/4q$ obtaining $\X'\subseteq \{y/\|y\|_\X: y\in \Y\}$ 
 of cardinality $\kappa$ such that
 $\X'$ is $(1+\varepsilon)$-concentrated in the $\|\ \|_\X$ norm. 
 By Lemma \ref{spheres} for $a=q-(1-q)/4\leq\|y\|_\X\leq\|y'\|_\X\leq q=b$ we obtain that 
 $$\|y'-y''\|_\X\leq q(1+(1-q)/4q)+(1-q)/4=$$
 $$= (q+(1-q)/4)+(1-q)/4=q+(1-q)/2 $$
 for every $y', y''\in\Y'=\{y\in \Y: y/\|y\|_\X\in\X'\}$.
 Again using the countable cofinality
 of $\kappa$ and Lemma \ref{separable} we find $\Y''\subseteq \Y'$ of cardinality $\kappa$ such that 
 $\|T(y)-T(y')\|_\Z\leq (1-q)/4$ for all $y, y'\in \Y''$. So
 $$\|y-y'\|_T=\|y-y'\|_\X+\|T(y)-T(y')\|_\Z\leq q+3(1-q)/4=1-(1-q)/4$$
 obtaining that  $\Y''$ is $(1-\delta)$-concentrated for $\delta=(1-q)/4$ as required.
 \end{proof}

\section{Concentration in  spheres of Banach spaces $(\X_\A, \|\ \|_\infty)$ 
induced by $\R$-embeddable almost disjoint families}\label{ad}

\begin{definition}\label{def-Rembed} An almost disjoint family $\A$ of 
subsets of $\N$ is called $\R$-embeddable if there is a function
$\phi:\N\rightarrow\R$ such that the sets $\phi[A]$ for $A\in \A$ are ranges of sequences
converging to distinct reals. 
\end{definition}

So $\R$-embeddable families of cardinality continuum
are examples of perhaps most standard
almost disjoint families.

\begin{lemma}[Folklore]\label{exists-ad} There exist almost disjoint families of infinite subsets of $\N$
of cardinality continuum which are $\R$-embeddable.
\end{lemma}

\begin{remark} In fact,
an almost disjoint family $\A$ of infinite subsets of $\N$ is $\R$-embeddable
if and only if there is  an injective $\phi:\N\rightarrow \Q$ 
such that the sets $\phi[A]$ for $A\in A$ are ranges of sequences
converging to distinct irrational reals. This follows from 
 Lemma 1 and Lemma 2 of \cite{akemann}.
\end{remark}

\begin{lemma}\label{ad} Suppose that $\kappa$ is a  cardinal of uncountable cofinality
not bigger than continuum, $X\subseteq [0,1]$ is uncountable and that 
$\A=\{A_x: x\in X\}$   is an $\R$-embeddable almost disjoint family of infinite subsets of $\N$.
  Then
\noindent given
\begin{enumerate}
\item $k\in\N$,
\item a finite $F\subseteq \N$,
\item a collection $\{a_\xi: \xi<\kappa\}$
of pairwise disjoint finite subsets of $X$ with $a_\xi=\{x_1^\xi, ..., x_k^\xi\}$ 
($x_i^\xi\not=x^\xi_j$ for $1\leq i<j\leq k$)
for any $\xi<\kappa$ 
  such that
$$A_{x_i^\xi}\cap A_{x_j^\xi}\subseteq F$$
for any $\xi<\kappa$ and any $1\leq i<j\leq k$
\end{enumerate}
there is a subset $\Gamma\subseteq \kappa$ of cardinality $\kappa$
such that for every  $\xi,\eta\in \Gamma$
we have 
$$A_{x_i^\xi}\cap A_{x_j^\eta}\subseteq F$$
for every  $1\leq i<j\leq k$.
\end{lemma}
\begin{proof} 
Let $\phi:\N\rightarrow \R$ be as in the definition of $\R$-embeddability.
By composing it with a homeomorphism from $\R$ onto $(0,1)$ we may assume that
$Y=\phi[\N]\subseteq [0,1]$. As the properties stated in the lemma do not change if we relabel
the elements of $\A$, we may assume that 
  $\phi[A_x]=\{q_n^x: n\in \N\}\subseteq[0,1]\cap Y$ is
such that $(q_n^x)_{n\in \N}$ converges to $x$. 

Fix $k\in \N$ and a finite $F\subseteq \N$. Let $\{a_\xi: \xi<\omega_1\}$ be
a collection
of pairwise disjoint finite subsets of $X$ 
with $a_\xi=\{x_1^\xi, ..., x_k^\xi\}$ as in the Lemma. Using the uncountable
cofinality of $\kappa$, by passing to
 subset of cardinality $\kappa$ we may assume that for all $\xi<\kappa$ we have
$\delta_\xi>\delta$ for some $\delta>0$, where
$$\delta_\xi=\min(\{|x_i^\xi-x_j^\xi|: 1\leq i<j\leq k\}.$$
Now note that there is
$(x_1, \dots, x_k)\in [0,1]^k$ such that every Euclidean neighbourhood
of $(x_1, \dots, x_k)$ contains $\kappa$-many points $(x_1^\xi, ..., x_k^\xi)$
for $\xi<\kappa$. This is because otherwise we can cover $[0,1]^k$ by open
sets each containing less than $\kappa$-many of the points $(x_1^\xi, ..., x_k^\xi)$ and 
the existence of a finite subcover 
would contradict the fact that all
$(x_1^\xi, ..., x_k^\xi)$s are distinct as the sets $\{x_1^\xi, ..., x_k^\xi\}$
are pairwise disjoint.

Since $\delta_\xi>\delta$ for every $\xi<\kappa$ we  have that
$\min(\{|x_i-x_j|: 1\leq i<j\leq\}\geq\delta$, and in particular all
$x_i$s for $1\leq i\leq k$ are distinct.
Let $I_1, \dots I_k$ be open intervals such that $x_i\in I_i$ and $I_i\cap I_j=\emptyset$
for all $1\leq i<j\leq k$. By the choice of $(x_1, \dots, x_k)$
there is  $\Gamma'\subseteq \kappa$ of cardinality $\kappa$ such that
$$x_i^\xi\in I_i$$
for all $\xi\in \Gamma'$  and all $1\leq i\leq k$.

Since the sequence $(q_n^{x_i^\xi})_{n\in\N}$ converges to ${x_i^\xi}$,
for every $\xi\in \Gamma'$ there is a sequence $(F_1^\xi, \dots F_k^\xi)$ of finite subsets
of $\N$ such that for each $1\leq i\leq k$ we have 
\begin{enumerate}
\item $F_i^\xi\subseteq A_{x^i_\xi}$
\item $\phi[A_{x^i_\xi}\setminus F_i^\xi]\subseteq I_i$
\end{enumerate}
As $I_i\cap I_j=\emptyset$ for distinct $i, j\leq k$  condition (2) implies that
\begin{enumerate}[(3)]
\item $(A_{x^i_\xi}\setminus F_i^\xi)\cap (A_{x^j_\eta}\setminus F_j^\eta)=\emptyset$
for any $\xi, \eta<\kappa$ and  $1\leq i<j\leq k$.
\end{enumerate}
As the cofinality of $\kappa$ is uncountable and
there are countably many of such $(F_1^\xi, \dots F_k^\xi)$s (as they are all subsets
of  $\N$) we may
 choose $\Gamma\subseteq \Gamma'$ of cardinality $\kappa$ such that
$(F_1^\xi, \dots F_k^\xi)$ is equal to some fixed $(F_1, \dots F_k)$. 

It remains to prove the property of $\Gamma$ stated in the Lemma. Fix distinct
$\xi, \eta\in \Gamma$ and $1\leq i<j\leq k$. Then by (1) and (3)
for all $1\leq i<j\leq k$ we have
$$A_{x_i^\xi}\cap A_{x_j^\eta}=(F_i\cap F_j)\cup (A_{x_i^\xi}\cap F_j)\cup (A_{x_j^\eta}\cap F_i).$$
By (1) this set is included in $(A_{x_i^\xi}\cap A_{x_j^\xi})\cup (A_{x_j^\eta}\cap A_{x_i^\eta})$
which is included in $F$ by the hypothesis of the lemma. So
we obtain $A_{x_i^\xi}\cap A_{x_j^\eta}\subseteq F$ as required.

\end{proof}

\begin{proposition}\label{Rembed} Suppose
that $\A$ is an $\R$-embeddable  almost disjoint family of subsets of $\N$. 
Whenever $\varepsilon>0$ and
 $\Y$ is a subset of
of the unit sphere of $(\X_\A, \|\ \|_\infty)$  of regular uncountable cardinality $\kappa$, there is 
$\Y'\subseteq \Y$ of cardinality $\kappa$
which is $(1+\varepsilon)$-concentrated.
\end{proposition}
\begin{proof} 

We may label elements of $\A$ as $A_x$s for $x\in X$  for some $X\subseteq[0,1]$.

Fix a subset $\Y=\{f_\xi: \xi<\kappa\}$ 
of the unit sphere of $(\mathcal X_\A,\|\ \|_\infty)$ of uncountable regular cardinality $\kappa$.
By the definition of $\X_\A$ for each $\xi<\kappa$ there are finite sets
$\{x^\xi_1, \dots x^\xi_{m_\xi}\}\subseteq X$ ($x^\xi_i\not=x^\xi_j$ for
 $1\leq i< j\leq m_\xi$ and $\xi<\kappa$), distinct rationals
$q^\xi_1, \dots q^\xi_{m_\xi}$ for all $\xi<\kappa$
and  rational valued finitely supported  $g_\xi\in c_0$ 
such that
$$\Big\|f_\xi-\big(g_\xi+\sum_{1\leq i\leq m_\xi} q^\xi_i 1_{A_{x^\xi_i}}   \big)\Big\|_\infty
\leq \varepsilon/3.$$
Using the uncountable cofinality
of $\kappa$ by passing to a subset of cardinality $\kappa$ we may assume that for all $\xi<\omega_1$ we have
$m_\xi=m$ for some $m\in \N$, $g_\xi=g$ 
for some finitely supported $g\in c_0$, and
that $q^\xi_i=q_i$ for some rationals $q_i$ for all $1\leq i\leq m$ and for all $\xi<\kappa$.
Since for a regular cardinal $\kappa$
any family of $\kappa$-many finite sets contains a subfamily of cardinality $\kappa$
such that the intersection of any two distinct members of the subfamily is a fixed 
finite $\Delta\subseteq\kappa$
(A version of the $\Delta$-system Lemma, Ex 25.3 of \cite{komjath})
by passing to a subset of cardinality $\kappa$ we may assume that there is such a $\Delta$
satisfying
$$\{x^\xi_1, \dots x^\xi_{m}\}\cap \{x^\eta_1, \dots x^\eta_{m}\}=\Delta$$
for every $\xi<\eta<\kappa$. By passing to a subset of cardinality $\kappa$
we may assume that there is a fixed $G\subseteq\{1, \dots, m\}$ such that
$\Delta=\{x^\xi_i: i\in G\}$ for every $\xi<\kappa$. 
If $G=\{1, \dots, m\}$, we conclude that $\{f_\xi: \xi<\kappa\}$  is even $2\varepsilon/3$-concentrated,
so we will assume that $G$ is a proper subset of $\{1, \dots, m\}$.
So by reordering $\{1, \dots, m\}$ we may  assume  that 
there is $1\leq k\leq m$ such that the sets $\{x^\xi_1, \dots x^\xi_{k}\}$ are pairwise disjoint
for $\xi<\kappa$ and $\{x^\xi_{k+1}, \dots x^\xi_{m}\}=\{x_{k+1}, \dots x_m\}=\Delta$ for some
$x_{k+1}, \dots x_m\in X$ for each $\xi<\kappa$.   So for every $\xi<\kappa$ we have
$$\Big\|f_\xi-\big(g+\sum_{1\leq i\leq m} q_i 1_{A_{x^\xi_i}}   
\big)\Big\|_\infty\leq \varepsilon/3.\leqno (1)$$
And for every $\xi<\eta<\kappa$ we have
$$\|f_\xi-f_\eta\|_\infty\leq \|\sum_{1\leq i\leq k} q_i 1_{A_{x^\xi_i}}-\sum_{1\leq i\leq k} 
q_i 1_{A_{x^\eta_i}}\|_\infty+2\varepsilon/3.\leqno (2)$$
Also as for each $1\leq j\leq k$ we have $\lim_{n\in A_{x^\xi_j}}|f_\xi(n)|\leq 1$ since 
$f_\xi$s are taken from the unit sphere in the $\|\ \|_\infty$-norm, and $\lim_{n\in A_{x^\xi_j}}g(n)=0$
as $g\in c_0$ and
$$\lim_{n\in A_{x^\xi_j}}\Big(\sum_{1\leq i\leq m} q_i 1_{A_{x^\xi_i}}\Big)(n)=q_j,$$
by Lemma \ref{functions-ad}, so we may conclude from (1) that for all $1\leq j\leq k$ we have 
$$|q_j|\leq 1+\varepsilon/3.\leqno (3)$$
Moreover, since there are countably many
finite subsets of $\N$ and $\kappa$ is of uncountable cofinality, we may assume that there is a
fixed  finite $F\subseteq \N$ such that
\begin{enumerate}[(4)]
\item $A_{x^\xi_i}\cap A_{x^\xi_j}\subseteq F$ for all $1\leq i<j\leq m$ and every $\xi<\kappa$.
\end{enumerate}
Using Lemma \ref{ad} and the hypothesis that $\A$ is $\R$-embeddable 
there is  subset $\Gamma\subseteq\kappa$
of cardinality $\kappa$ such  that
\begin{enumerate}[(5)]
\item $A_{x^\xi_i}\cap A_{x^\eta_j}\subseteq F$  for any $1\leq i<j\leq k$ and
any $\xi, \eta\in \Gamma$.
\end{enumerate}
  Since there are countably many rational valued functions 
defined on $F$, by passing to a subset of $\Gamma$ of cardinality $\kappa$ we may also assume that 
for each $\xi, \eta\in \Gamma$ and each $n\in F$ we have 
$$\big(\sum_{1\leq i\leq k} q_i 1_{A_{x^\xi_i}}\big)(n)=
\big(\sum_{1\leq i\leq k} q_i 1_{A_{x^\eta_i}}\big)(n).$$
So by  (5) the following are the only possible cases for $n\in\N$:
$$
  \Big(\sum_{1\leq i\leq k} q_i 1_{A_{x^\xi_i}}-\sum_{1\leq i\leq k} 
q_i 1_{A_{x^\eta_i}}\Big)(n)
 =
  \begin{cases}
    0 & \text{if $n\in F$,} \\
    q_i-q_i=0 & \text{if $n\in A_{x^\xi_i}\cap A_{x^\eta_i}\setminus F$ for some $i$,}\\
    q_i& \text{if $n\in A_{x^\xi_i}\setminus (A_{x^\eta_i}\cup F)$ for some $i$,}\\
    -q_i& \text{if $n\in A_{x^\eta_i}\setminus (A_{x^\xi_i}\cup F)$ for some $i$,}\\
    0& \text{if $n\not\in \bigcup_{1\leq i\leq k}(A_{x^\xi_i}\cup A_{x^\eta_i})$.}
  \end{cases}
$$
By (3) it follows that
$$\Big\|\sum_{1\leq i\leq k} q_i 1_{A_{x^\xi_i}}-\sum_{1\leq i\leq k} 
q_i 1_{A_{x^\eta_i}}\Big\|_\infty\leq \max\{|q_i|: 1\leq i\leq k\}\leq 1+\varepsilon/3$$
Which by (2) implies the required $\|f_\xi-f_\eta\|\leq 1+\varepsilon$ for any $\xi, \eta\in \Gamma$.
\end{proof}

\begin{remark} We note that in the language of the paper \cite{nLuzin} of
 Hru\v s\'ak and Guzm\'an the property of an almost disjoint family $\A$ from Lemma \ref{ad}
 means  that $\A$ contains no $n$-Luzin gap for any $n\in \N$. They showed
 that under the assumption of   {\sf MA} and the negation of 
 {\sf CH} every
 almost disjoint family of cardinality smaller than continuum which
 contains no $n$-Luzin gaps for any $n\in \N$ is  $\R$-embeddable.
 \end{remark}
 
 \begin{remark} By Proposition \ref{Rembed} for every $\varepsilon>0$ the space $(\X_\A, \|\ \|_\infty)$
 for an $\R$-embeddable almost disjoint family $\A$ admits no
 uncountable $(1+\varepsilon)$-separated set in its unit sphere. Nevertheless
 such spaces always admit $(1+)$-separated sets of unit vectors of the cardinality equal to
 the density of $\X_\A$.
 For $A\in \A$ define 
 $$f_A=1_A-\sum\{{1\over{k+1}}1_{\{k\}}: k\in \N\setminus A\}$$
 Then given any two distinct $A, A'\in \A$, choose $k\in A\setminus A'$
 and observe that we have $\|f_A-f_{A'}\|\geq |f_A-f_{A'}|(k)=1-(-1/({k+1}))>1$.
 \end{remark}

\section{Separation in  spheres of Banach spaces $(\X_\A, \|\ \|_T)$ 
induced by Luzin almost disjoint families}\label{luzin}

\begin{definition}\label{def-luzin} An almost disjoint family $\{A_\xi: \xi<\omega_1\}$ is called a Luzin family
if $f_\alpha: \alpha\rightarrow \N$ is finite-to-one for any $\alpha<\omega_1$, where
$$f_\alpha(\beta)=\max(A_\beta\cap A_\alpha)$$
for every $\beta<\alpha<\omega_1$.
\end{definition}

\begin{lemma}[\cite{luzin}] Luzin families exist.
\end{lemma}

Luzin families were first constructed in \cite{luzin}.
See Section 3.1. of \cite{hrusak-ad} for the construction and for more information.

\begin{proposition}\label{luzin-2} Suppose that $\mathcal L=\{A_\xi: \xi<\omega_1\}$ is  a Luzin family, 
and $\xi_\alpha, \eta_\alpha$  satisfy
$$\xi_\beta<\eta_\beta<\xi_\alpha<\eta_\alpha<\omega_1$$
for all $\beta<\alpha<\omega_1$.
Then $(\X_{\mathcal L}, \|\ \|_\infty)$ admits  an uncountable $2$-equilateral set $\{f_\alpha: \alpha\in X\}$
among elements
of its unit sphere of the form
$$f_\alpha=1_{A_{\xi_\alpha}}-1_{A_{\eta_\alpha}}$$
for all $\alpha\in \Gamma$  for some uncountable $\Gamma\subseteq\omega_1$.
\end{proposition}
\begin{proof}
It is clear that $\|f_\alpha\|_\infty\leq 1$ for each $\alpha<\omega_1$ 
as $f_\alpha$s may assume values among $\{-1,0, 1\}$. Since $A_{\xi_\alpha}\cap A_{\eta_\alpha}$
is finite and both $A_{\xi_\alpha}, A_{\eta_\alpha}$ are infinite, we have that $\|f_\alpha\|=1$
for each $\alpha<\omega_1$.
Let $F_\alpha\subseteq\N$ for $\alpha<\omega_1$ be such a finite set that 
$A_{\xi_\alpha}\cap A_{\eta_\alpha}\subseteq F_\alpha$. By passing to an uncountable subset
we may assume that $F_\alpha=F$ for each $\alpha<\omega_1$ and some finite $F\subseteq \N$.
Let $k\in\N$ be such that $F\subseteq \{1, \dots, k\}$. Now we will use the following version of Erd\"os-Dushnik-Miller theorem: 
whenever $\kappa$ is regular and uncountable cardinal  and $c:[\kappa]^2\rightarrow \{0,1\}$,
then either there is $1$-monochromatic set of cardinality $\kappa$ or
there is a $0$-monochromatic set of order type $\omega+1$ (see 24.32 of \cite{komjath}).
Consider a coloring $c:[\omega_1]^{2}\rightarrow \{0,1\}$ defined by
$c(\{\beta, \alpha\})=0$ if  $f_{\xi_\alpha}(\eta_\beta)\leq k$ 
for the ordering of the elements $\beta<\alpha$ and
$c(\{\beta, \alpha\})=1$ otherwise. Note that there cannot be a $0$-monochromatic set
$\Delta\subseteq\omega_1$ of order type $\omega+1$ 
because  denoting its biggest element as $\alpha$ we would 
have that $f_{\xi_\alpha}|\{\eta_\beta: \beta \in \Delta\cap\alpha\}$ is bounded 
below $k$ which would contradict the hypothesis that
$\mathcal L$ is a Luzin family i.e., $f_{\xi_\alpha}$s are finite-to-one. So it follows that
there is an uncountable $1$-monochromatic $\Gamma\subseteq \omega_1$ for the coloring $c$.

For $\alpha, \beta\in \Gamma$ with $\beta<\alpha$ 
there is $m>k$ such that $m\in A_{\xi_\alpha}\cap A_{\eta_\beta}$.
By the choice of $F$ and the fact that $F\subseteq \{1, \dots, k\}$ we have that
neither $m\in A_{\eta_\alpha}$ nor $m\in A_{\xi_\beta}$ and so
it follows that 
$$\|f_\alpha-f_\beta\|_\infty=\|(1_{A_{\xi_\alpha}}-1_{A_{\eta_\alpha}})
-(1_{A_{\xi_\beta}}-1_{A_{\eta_\beta}})\|_\infty=2.$$
\end{proof}

\begin{proposition}\label{luzin-renorm} Suppose that $\mathcal L=\{A_\xi: \xi<\omega_1\}$ is  a Luzin family, 
$\A$ is an almost disjoint family such that $\mathcal L\subseteq \A$,  $\X$ is a Banach space
and $T:\mathcal X_\A\rightarrow \X$ is a bounded linear operator with separable range.
Then the unit sphere of
$(\X_\A, \|\ \|_T)$ admits an uncountable $(1+{1\over 2})$-separated set.
\end{proposition}
\begin{proof} As $\X_{\mathcal L}$ is a subspace of $\X_\A$ we can assume that $\A=\mathcal L$.
By Lemma \ref{separable} there is an uncountable
subset $\Xi\subseteq \omega_1$ such that
$$\|T(1_{A_\xi}-1_{A_\eta})\|\leq 1/3\leqno (*)$$
for every $\xi, \eta\in \Xi$. 
Let $\{\xi_\alpha: \alpha<\omega_1\}$
and $(\eta_\alpha: \alpha<\omega_1)$  be enumerations of two uncountable and disjoint subsets of $\Xi$
satisfying for all $\beta<\alpha<\omega_1$  the following conditions
$$\xi_\beta<\eta_\beta<\xi_\alpha<\eta_\alpha.$$
By Proposition \ref{luzin-2} there is an uncountable  $\Gamma\subseteq\omega_1$
such that   $\{f_\alpha: \alpha\in \Gamma\}$ is a $2$-equilateral set
of the unit sphere of $(\X_\A, \|\ \|_\infty)$, where for each $\alpha\in \Gamma$ we have
$$f_\alpha=1_{A_{\xi_\alpha}}-1_{A_{\eta_\alpha}}.$$
For $\alpha\in \Gamma$ consider
$$g_\alpha=f_\alpha/\|f_\alpha\|_T.$$
We claim that $\{g_\alpha: \alpha\in \Gamma\}$ is the desired
$(1+1/2)$-separated set of the unit sphere of $(\X_\A, \|\ \|_T)$.
To prove it take $\beta<\alpha$ with $\alpha, \beta\in \Gamma$.
As $$\|g_\alpha-g_\beta\|_T=\|g_\alpha-g_\beta\|_{\infty}+\|T(g_\alpha)-T(g_\beta)\|_\X,$$
it enough to prove that
$\|g_\alpha-g_\beta\|_{\infty}\geq 1+1/2$. We have
$$\|g_\alpha-g_\beta\|_{\infty}=
\|{{f_\alpha}\over{\|f_\alpha\|_T}}-{{f_\alpha}\over{\|f_\alpha\|_T}}\|_\infty
\geq \|f_\alpha-f_\beta\|_\infty-\|{{f_\alpha}\over{\|f_\alpha\|_T}}-f_\alpha\|_\infty
-\|{{f_\beta}\over{\|f_\beta\|_T}}-f_\beta\|_\infty.$$

Since $\alpha, \beta\in \Gamma$ we have $\|f_\alpha\|_\infty=\|f_\beta\|_\infty=1$,
$\|f_\alpha-f_\beta\|_\infty=2$.
So now we estimate 
$$\|{{f_\alpha}\over{\|f_\alpha\|_T}}-f_\alpha\|_\infty=
\Big| {1\over{\|f_\alpha\|_\infty+\|T(f_\alpha)\|_\X}}
 -1\Big|\|f_\alpha\|_\infty=\Big| {1\over{1+\|T(f_\alpha)\|_\X}} -1\Big|.
$$
By (*) we have $\|T(f_\alpha)\|_\X\leq 1/3$, and so
$$3/4= {1\over{1+1/3}}\leq {1\over{1+\|T(f_\alpha)\|_\X}} \leq 1,$$
and so
$$\|{{f_\alpha}\over{\|f_\alpha\|_T}}-f_\alpha\|_\infty\leq 1/4.$$
The same calculation works for $\|{{f_\beta}\over{\|f_\beta\|_T}}-f_\beta\|_\infty,$
so we conclude that 
$$\|g_\alpha-g_\beta\|_{\infty}\geq 2-1/4-1/4= 1+1/2$$ as required.
\end{proof}

\section{Final remarks and open problems}

Recall that the Ramsey theorem says that given any $k\in \N$
and any coloring $c: [\N]^2\rightarrow \{1, \dots, k\}$ there is
an infinite $A\subseteq \N$ which is $i$-monochromatic for $c$ for some $1\leq i\leq k$, that
is $c[[A]^2]=\{i\}$. On the other hand colorings of all pairs of uncountable cardinals 
not bigger than  continuum may not have uncountable monochromatic sets as already shown by Sierpi\'nski
(24.23  of \cite{komjath}).
The phenomena considered in this paper can be seen from a Ramsey theoretic point of view.
For example, given a Banach space $X$ one can  consider a coloring
$c:[S_{\X}]^2\rightarrow \{-1,0,1\}$ given by 
$$
  c_\X(x, y)=
  \begin{cases}
    -1 & \text{if $\|x-y\|<1$} \\
    0 & \text{if $\|x-y\|=1$} \\
    1 & \text{if $\|x-y\|>1$}
  \end{cases}
$$
A$(1+)$-separated set in $S_{\X}$ is a $1$-monochromatic set for $c_\X$ and 
by an observation of Terenzi mentioned in the introduction the
existence of an uncountable equilateral set in $\X$ is equivalent to the existence
of an uncountable $0$-monochromatic set for $c_\X$. On the other hand our main result says that
there is a nonseparable Banach space $\X$ such that every uncountable subset of
$S_{\X}$ contains an uncountable ($-1$)-monochromatic set. It is not accidental
that the separable results of Kottman or
 Elton and Odell employ the Ramsey theorem (\cite{kottman, elton-odell}).
 However, we are still far from understanding the structure
 of monochomatic sets for the colorings
 $c_\X$. For example, it is natural to ask if the following dichotomy holds:
 
 \begin{question} Is it true that the unit sphere of every nonseparable Banach space $\X$ 
 either contains a $(1+)$-separated set or every uncountable subset
 of the sphere $S_\X$ contains an uncountable subset which is $(1-)$-concentrated
 (or $(1-\varepsilon)$-concentrated for some $\varepsilon>0$)?
 \end{question}
 Here by $(1-)$-concentrated we mean a set $\Y\subseteq \X$ such that $\|x-y\|<1$
 for any two distinct $x, y\in \Y$. 
 In fact, if the above
 dichotomy does not hold, it would be interesting to search for
 a nonseparable Banach space $\X$ such that $c_\X$ has no uncountable
 $1$-monochromatic sets and the family of all uncountable $(-1)$-monochromatic sets is minimal in some sense.
 Another aspect of our paper is the following general question:
 
 \begin{question} What nonseparable metric spaces can be isometrically embedded in every
 nonseparable Banach space?
 \end{question}
 
 The Riesz lemma implies that for every $\varepsilon>0$ 
 every nonseparable Banach space of density $\kappa$ contains
 a metric space $(M, d)$, where $M=\{0\}\cup\{x_\xi: \xi<\kappa\}$, and where
 $d(0, x_\xi)=1$ for all $\xi<\kappa$ and $d(x_\xi, x_\eta)>1-\varepsilon$
 for  any $\xi<\eta<\kappa$. In this article, among other results, we have shown  in ZFC that uncountable
 metric spaces, where distances between any two distinct points are the same, do not 
 embed isometrically into all nonseparable Banach spaces.
 The above question has been considered on the countable
 level in \cite{mer-iso}. On the finite level, for example    Shkarin  has proved in \cite{shkarin}
 that every finite ultrametric space (a metric space where the distance $d$
 satisfies $d(x, z)\leq\max(d(x, y), d(y, z))$ for any points $x, y, z$) 
 isometrically embeds in any infinite dimensional Banach space.

\bibliographystyle{amsplain}

\begin{thebibliography}{120}

\bibitem{papini1} M.  Baronti, E.  Casini, P. L.  Papini, 
\emph{On average distances and the geometry of Banach spaces.}
Nonlinear Anal. 42 (2000), no. 3, Ser. A: Theory Methods, 533--541. 

\bibitem{sailing} F. Cabello S\'anchez, J.  Castillo, W.  Marciszewski, G.  Plebanek,  A. Salguero-Alarc\'on, \emph{Sailing over three problems of Koszmider}. J. Funct. Anal. 279 (2020), no. 4, 108571.

\bibitem{castillo}  J. Castillo, M. Gonz\'alez, P. Papini, \emph{New results on Kottman's constant}. Banach J. Math. Anal. 11 (2017), no. 2, 348--362. 

\bibitem{cuth} M. C\'uth, O. Kurka, B. Vejnar, \emph{Large separated sets of unit vectors in
Banach spaces of continuous functions}, Colloq. Math. 157 (2019), no. 2, 173--187.

\bibitem{day} M. Day, 
On the basis problem in normed spaces.
Proc. Amer. Math. Soc. 13 (1962), 655--658. 

\bibitem{dgz} R. Deville, G.  Godefroy, V.  Zizler, \emph{Smoothness and renormings in Banach spaces}. Pitman Monographs and Surveys in Pure and Applied Mathematics, 64. Longman Scientific \& Technical, Harlow; copublished in the United States with John Wiley \& Sons, Inc., New York, 1993.

\bibitem{dow} A. Dow, \emph{Sequential order under PFA}. Canad. Math. Bull. 54 (2011), no. 2, 270--276.

\bibitem{elton-odell} J. Elton and E. Odell, \emph{The unit ball of every infinite-dimensional normed linear space contains a $(1+\varepsilon)$-separated sequence}, Colloq. Math. 44 (1981), 105--109. 

\bibitem{unif-convex} D. Freeman, E.  Odell, B. Sari, T.  Schlumprecht, 
\emph{Equilateral sets in uniformly smooth Banach spaces}. Mathematika 60 (2014), no. 1, 219--231. 

\bibitem{mer-serdica} E. Glakousakis, S.  Mercourakis, 
\emph{Examples of infinite dimensional Banach spaces without infinite equilateral sets}. 
Serdica Math. J. 42 (2016), no. 1, 65--88. 



\bibitem{godun2} B. V. Godun, 
\emph{A special class of Banach spaces}. 
Mat. Zametki 37 (1985), no. 3, 391--398, 462. 

\bibitem{godun}  B. V. Godun, \emph{Auerbach bases in Banach spaces that are isomorphic 
to $\ell_1[0,1]$}.  C. R. Acad. Bulgare Sci. 43 (1990), no. 10, 19--21 (1991).

\bibitem{troyanski} B. V. Godun, B. Lin, S. Troyanski, 
\emph{On Auerbach bases}.  Banach spaces (M\'erida, 1992), 115--118,
Contemp. Math., 144, Amer. Math. Soc., Providence, RI, 1993. 

\bibitem{guirao} A. Guirao, V. Montesinos, V. Zizler, \emph{Open problems in the geometry and analysis of Banach spaces}. Springer, [Cham], 2016. 

\bibitem{akemann}  O. Guzm\'an, M. Hru\v s\'ak, P. Koszmider, \emph{On $\R$-embeddability of almost disjoint families and Akemann-Doner C*-algebras}.  Fund. Math. 254 (2021), no. 1, 15--47.

\bibitem{hajek-book} P. H\'ajek, V. Montesinos Santaluc\'\i a, J. Vanderwerff, V. Zizler, \emph{Biorthogonal
systems in Banach spaces}, CMS Books in Mathematics/Ouvrages de Math\'ematiques de la
SMC, vol. 26, Springer, New York, 2008.

\bibitem{hajek-tams}  P. H\'ajek, T. Kania, T. Russo, \emph{Separated sets and Auerbach systems in Banach spaces}. Trans. Amer. Math. Soc. 373 (2020), no. 10, 6961--6998.

\bibitem{nLuzin} M. Hru\v s\'ak, O.  Guzm\'an, \emph{n-Luzin gaps}. 
Topology Appl. 160 (2013), no. 12, 1364--1374.

\bibitem{hrusak-ad} M. Hru\v s\'ak, \emph{Almost disjoint families and topology}. Recent progress in general topology. III, 601--638, Atlantis Press, Paris, 2014. 

 
\bibitem{jech}  T. Jech, \emph{Set theory.} The third millennium edition, revised and expanded. Springer Monographs in Mathematics. Springer-Verlag, Berlin, 2003.


\bibitem{jl} W. Johnson, J.  Lindenstrauss, \emph{Some remarks on weakly compactly generated Banach spaces}. Israel J. Math. 17 (1974), 219--230. 

\bibitem{jl-cor} W. Johnson, J.  Lindenstrauss, \emph{ Correction to: "Some remarks on weakly compactly generated Banach spaces''}  Israel J. Math. 32 (1979), no. 4, 382--383. 

\bibitem{tt} T. Kania, T. Kochanek, \emph{Uncountable sets of unit vectors that are separated
by more than $1$}, Studia Math. 232 (2016), no. 1, 19--44.

\bibitem{komjath} P. Komj\'ath, V.  Totik, 
\emph{Problems and theorems in classical set theory}.
Problem Books in Mathematics. Springer, New York, 2006. 

\bibitem{equi} P.  Koszmider, \emph{Uncountable equilateral sets in Banach spaces of the form $C(K)$}. 
Israel J. Math. 224 (2018), no. 1, 83--103. 

\bibitem{supermrowka} P. Koszmider, N. J.  Laustsen, 
\emph{A Banach space induced by an almost disjoint family, admitting only few operators and decompositions}. (English summary)
Adv. Math. 381 (2021), 107613. 

\bibitem{equi2} P. Koszmider, H. M. Wark, \emph{Large Banach spaces with no infinite equilateral sets}. 
{\tt arXiv:2104.05330}.

\bibitem{kottman} C. Kottman, 
\emph{Subsets of the unit ball that are separated by more than one}.
Studia Math. 53 (1975), no. 1, 15--27. 

\bibitem{kottman2} C. Kottman, 
\emph{Packing and reflexivity in Banach spaces.}
Trans. Amer. Math. Soc. 150 (1970), 565--576. 

\bibitem{kryczka} A. Kryczka, S. Prus, \emph{Separated sequences in nonreflexive Banach spaces}. Proc. Amer. Math. Soc. 129 (2001), no. 1, 155--163. 


\bibitem{luzin} N. N. Luzin, \emph{On subsets of the series of natural numbers}. Izvestiya Akad. Nauk SSSR. Ser.Mat. 11 (1947), 403--410

\bibitem{mar-pol} W. Marciszewski, R. Pol, \emph{On Banach spaces whose norm-open sets are 
$F_\sigma$-sets in the weak topology}. J. Math. Anal. Appl. 350 (2009), no. 2, 708--722.

\bibitem{mer-pams} S. Mercourakis, G.  Vassiliadis, 
\emph{Equilateral sets in infinite dimensional Banach spaces. }
Proc. Amer. Math. Soc. 142 (2014), no. 1, 205--212. 

\bibitem{mer-ck}  S. Mercourakis, G.  Vassiliadis, \emph{Equilateral sets in Banach spaces of
 the form $C(K)$}. Studia Math. 231 (2015), no. 3, 241--255. 
 
\bibitem{mer-iso}  S. Mercourakis, G. Vassiliadis, \emph{Isometric embeddings of a class of separable metric spaces into Banach spaces}. Comment. Math. Univ. Carolin. 59 (2018), no. 2, 233--239. 
 
 \bibitem{unifconvex}  J. van Neerven, \emph{Separated sequences in uniformly convex Banach spaces}. Colloq. Math. 102 (2005), no. 1, 147--153. 
 

 
\bibitem{papini2} P. L.  Papini, 
\emph{Bad normed spaces, convexity properties, separated sets.}
Nonlinear Anal. 73 (2010), no. 6, 1491--1494. 
 
 \bibitem{pel}  A. Pe\l czy\'nski, Z. Semadeni, \emph{Spaces of continuous functions (III) (Spaces $C(\Omega)$ for $\Omega$ without perfect subsets)}, Studia Math. 18 (1959), 211--222.
 
 \bibitem{plicko} A. Pli\v cko, 
\emph{On bounded biorthogonal systems in some function spaces}.
Studia Math. 84 (1986), no. 1, 25--37.
 
 \bibitem{riesz} F. Riesz, \emph{\"Uber lineare Funktionalgleichungen}, Acta Math. 41 (1916), 71--98.
 
 \bibitem{shkarin} S. Shkarin, 
\emph{Isometric embedding of finite ultrametric spaces in Banach spaces}. 
Topology Appl. 142 (2004), no. 1-3, 13--17. 
 
 \bibitem{terenzi2} P. Terenzi \emph{Successioni regolari negli spazi di Banach} Rend. Sem. Mat. Fis. Milano 57, 1 (1987), 275--285. 
 
\bibitem{terenzi}  P. Terenzi, \emph{Equilater sets in Banach spaces}. Boll. Un. Mat. Ital.
 A (7) 3 (1989), no. 1, 119--124.
 
 

\end{thebibliography}

\end{document}